\documentclass[12pt,reqno]{amsart}
\usepackage{amsmath}
\usepackage{amssymb, latexsym, amsfonts, amscd, amsthm, mathrsfs, enumerate, esint}
\usepackage[usenames,dvipsnames]{color}
\usepackage[all]{xy}
\usepackage{graphicx}
\usepackage{epsfig}
\usepackage{hyperref}
\usepackage{fullpage}
\usepackage{setspace}

\numberwithin{equation}{section}

\definecolor{purple}{rgb}{0.9,0,0.8}

\definecolor{gray}{rgb}{0.7,0.7,0.7}

\newcommand{\abbr}[1]{{\sc\lowercase{#1}}}

\newtheorem{thm}{Theorem}[section]
\newtheorem{cor}[thm]{Corollary}
\newtheorem{lem}[thm]{Lemma}
\newtheorem{ppn}[thm]{Proposition}

\newtheorem{defn}[thm]{Definition}
\newtheorem{pbm}[thm]{Problem}
\theoremstyle{definition}

\newtheorem{remark}[thm]{Remark}
\newcommand{\beq}{\begin{equation}}
\newcommand{\eeq}{\end{equation}}

\newcommand{\ep}{\epsilon} 
\newcommand{\vep}{\varepsilon}

\newcommand{\bB}{\mathbb{B}}

\newcommand{\bD}{\mathbb{D}}
\newcommand{\E}{\mathbb{E}}
\newcommand{\bE}{\mathbb{E}}

\newcommand{\bG}{\mathbb{G}}
\newcommand{\G}{\mathbb{G}}

\newcommand{\bI}{\mathbb{I}}

\newcommand{\N}{\mathbb{N}}

	\renewcommand{\P}{\mathbb{P}}	
\newcommand{\bP}{\mathbb{P}}

\newcommand{\R}{\mathbb{R}}

\newcommand{\Z}{\mathbb{Z}}
\newcommand{\bZ}{\mathbb{Z}}

\newcommand{\cF}{\mathcal{F}}

\newcommand{\wt}{\widetilde}
\newcommand{\wh}{\widehat}

\renewcommand{\emptyset}{\varnothing}

\newcommand{\grad}{\nabla}

\begin{document}
\title[Transience in growing subgraphs via evolving sets
 ]
{Transience in growing subgraphs\newline via evolving sets
}
\date{\today}

\author[A.\ Dembo]{Amir Dembo$^*$}
\author[R.\ Huang]{Ruojun Huang$^\diamond$}
\author[B.\ Morris]{Ben Morris$^\star$}
\author[Y.\ Peres]{Yuval Peres$^\dagger$}
\address{$^*$Department of Mathematics, Stanford University, Building 380, 
Sloan Hall, Stanford, CA 94305, USA}
\address{$^{*\diamond}$Department of Statistics, Stanford University,
 Sequoia Hall, 390 Serra Mall, Stanford, CA 94305, USA}
\address{$^\star$ Department of Mathematics, University of California at Davis, One Shields Ave, Davis, CA 95616, USA}
\address{$^\dagger$ Theory Group, Microsoft Research, One Microsoft Way,
Redmond, WA 98052, USA}

\thanks{This research was supported in part by NSF grants DMS-1106627
and CNS-1228828.}



\maketitle
\begin{abstract}
We extend the use of 
random evolving sets to time-varying conductance models and utilize 
it to provide tight heat kernel upper bounds. It yields the transience of 
any uniformly lazy random walk, on $\bZ^d$, $d\ge3$, equipped 
with uniformly bounded above and below, independently time-varying 
edge conductances, of (effectively) non-decreasing in time 
vertex conductances, 
thereby affirming part of \cite[Conj. 7.1]{ABGK}.
\end{abstract}

\begin{section}{Introduction}
There has been much interest in random walks in random environment (see \cite{HMZ}). The challenge often comes from the highly non-reversible nature of the dynamics, which  can leave questions as fundamental as recurrence versus transience open. For example, the recurrence of linearly edge reinforced random walk with strong enough reinforcement strength on any graphs is just recently solved (\cite{ACK,ST,SZ}). 
Many questions in this general area are treated in an ad-hoc manner, and the development of methods in order to fully or partially resolve them is just as interesting as the questions themselves.  

The case when the evolution of the environment is independent of the stochastic process is better understood (e.g. \cite{DKL}), and there are conjectures on the emergence of universality (cf. \cite[Conj. 7.1]{ABGK} and \cite[Conj. 1.2, 1.8, 1.10]{DHS}). Specifically, \cite{DHS} conjecture that whenever a graph $\bG_\infty$ 
is recurrent, then any graph sequence $\{\bG_t\}_{t\in\N}$ dynamically growing towards 
$\bG_\infty$ is also recurrent, for the discrete time, simple random walk $\{X_t\}_{t\in\N}$ 
taking steps in $\{\bG_t\}_{t\in\N}$; and whenever $\bG_0$ is transient, then 
any growing sequence $\{\bG_t\}$ of uniformly bounded degrees,
starting from $\bG_0$ is transient.
 
Essentially the same phenomenon is conjectured
in \cite{ABGK} for the general setting of monotonically time varying 
conductance models, which are also the focus of the present work. That is,
the stochastic process $\{X_t\}_{t \in \N}$ on a locally finite 
graph $\bG=(V,E)$ constructed as 
random walk in time varying edge conductances $\{\pi^{(t)}(x,y), t \in \N, (x,y) \in E\}$ 
which are changed independently of the sample path $t \mapsto X_t$. 
Specifically, the vertex conductances 
\beq\label{eq:vert-cond-def}
\pi^{(t)}(x)=\sum_{y \in V} \pi^{(t)}(x,y)\, \qquad x \in V \,, 
\eeq 
form the time-dependent reversing measure for $X_t$, and setting
$V_t = \{x \in V: \pi^{(t)}(x) > 0 \}$, the transition probability 
of the 
in-homogeneous Markov chain $X_t \in V_t$ is given by
\beq\label{eq:trans-def} 
P(t,x;t+1,y)=\frac{\pi^{(t)}(x,y)}{\pi^{(t)}(x)} \,, \qquad \forall (x,y) \in E, 
\; x \in V_t\,.
\eeq 
When $\bG$ is a tree, \cite[Theorems 5.1]{ABGK} proves recurrence of such $\{X_t\}$
provided all edge conductances 
$\pi^{(t)}(x,y)$ 
are 
positive,
non-decreasing in $t$,
and bounded above by $\pi^{(\infty)}(x,y)$ of a time-invariant recurrent
model, 
while \cite[Theorem 5.2]{ABGK}
establishes its transience when all edge conductances are
positive, non-increasing in $t$
and bounded below by $\pi^{(\infty)}(x,y)$ of a transient 
model. 
Both results apply when $\bG=\N$,
for which they are complemented 
by \cite[Theorems 4.2 and 4.4]{ABGK} that cover also the non-decreasing 
transient and non-increasing recurrent cases. We note in passing that all
four theorems 
allow for non-Markovian processes, where 
edge conductances depend on the past trajectory of the walk, but 
\cite[Section 6]{ABGK} shows that in general (specifically, 
when $\bG=\Z^2$), these results may fail under such dependence. Nevertheless, 
\cite[Conj. 7.1]{ABGK} proposes that the aforementioned four theorems 
hold on any locally finite graph $\bG$, provided its time varying edge 
conductances are independent of the walk's trajectory (i.e. for the 
Markovian evolution as in \eqref{eq:trans-def}).

The present work affirms part of
the transient case of this conjecture
(and a special case of 
\cite[Conj. 1.8]{DHS}), 
for $\bZ^d$, $d\ge3$
equipped with uniformly bounded non-decreasing vertex conductances (more generally,
extending \cite[Theorem 4.2]{ABGK} from $\bG=\N$ to all graphs having 
suitable isoperimetric properties). In contrast, the recurrent direction 
(i.e. obtaining heat kernel lower bounds), is mostly open. 

We prove transience by way of establishing an on-diagonal heat kernel upper bound. The study of heat kernels for diffusions on manifolds and Markov chains on graphs has a long history, dating back at least to the work of De Giorgi, Nash, Moser in the late 1950s and early 60s, and that of Aronson (Cf. \cite{Ar}), investigating properties of solutions of parabolic differential equations. There is a large body of work on Gaussian and sub-Gaussian heat kernel estimates on diverse spaces, their equivalence to functional inequalities, and related stability theory (see\cite{CGZ,De,Gr,HS,SC,St1,St2} and the references therein). In the setting of graphs, some associated \emph{continuous time}, symmetric rate random walks among uniformly elliptic, time dependent conductances 
have been studied (cf. \cite[Section 4]{DD}, \cite[Appendix B]{GOS} and \cite[Theorem 1.1]{GP}). In particular, it is by now known that the two-sided Gaussian heat kernel estimates hold for any such random walks on $\Z^d$, and more generally on any bounded degree graphs satisfying volume doubling plus a uniform Poincar\'{e} inequality (cf. \cite[Theorem 1.2]{HK} and the references therein). 

All such continuous time, symmetric rate walks, have \emph{time-independent 
reversing measure.} Similarly if the discrete time-dependent
conductance model of \eqref{eq:trans-def} satisfies
a uniform Sobolev inequality, \cite[Section 7]{CGZ} claims 
some of the 
Gaussian heat kernel estimates, \emph{provided the reversing measure
$\pi^{(t)}(x)$ of \eqref{eq:vert-cond-def} is held constant in time,} and the walk is uniformly lazy.
In contrast, the study of recurrence/transience, and more generally, 
that of heat kernel estimates, is rather subtle when 
$t \mapsto \pi^{(t)}(x)$ is not constant. Indeed, some heat kernel
estimates are derived in this setting by \cite{SCZ}, but as shown in 
\cite[Propositions 1.4,1.5]{HK}, if the time varying vertex conductances
are either non-monotone or unbounded, then in general 
neither the upper/lower Gaussian estimates nor recurrence/transience 
properties are stable under perturbations (and the same applies
for constant speed continuous time random walks). 

Random evolving sets have been introduced in \cite{MP1,MP2}, where they 
are applied 
to study the mixing time of possibly non-reversible Markov chains 
(with the related notion of size-biased evolving sets already inherent in \cite{DF}). 
For static weighted graphs it is known that evolving sets serve well in 
deducing from an isoperimetric inequality, both the heat kernel upper bound 
and a Nash inequality. The main tool of this work is the extended notion of 
random evolving sets in the parabolic (time-varying) context 
(see Definition \ref{defn-evol}).

Turning to state our main result, we use hereafter 
$A^c$ for $V\backslash A$ and $\pi^{(t)}(A)=\pi^{(t)}(A,V)$, 
or more generally $\pi^{(t)}(A,B)=\sum_{x\in A, y\in B}\pi^{(t)}(x,y)$ 
for any $A \subset V_t$, $B \subset V$. 
By analogy to convention, we define the {\emph{heat kernel}} of $\{X_t\}$ as
\begin{align}\label{eq:ht-def}
h(s,x;t,y):=\frac{P(s,x;t,y)}{\pi^{(t)}(y)}, \quad x\in V_s,\, y\in V_t\,.
\end{align}

\begin{defn}\label{eff-non-dec}
Starting with $\beta(0)=1$, suppose that  
\begin{align}
\label{eq:eff-non-dec}
\beta(u+1) := \beta(u) \sup_{x\in V_{u}}
\Big\{\frac{\pi^{(u)}(x)}{\pi^{(u+1)}(x)}\big\} \,,
\quad
 u \in \N \,,
\end{align}
are finite. With $t\mapsto\beta(t)\pi^{(t)}(x)$ non-decreasing,
we call vertex conductances $t\mapsto\pi^{(t)}(x)$ 
\emph{effectively non-decreasing}, if $\eta_{\star} = 
\sup_{t>u \ge 0} \{\beta(t)/\beta(u)\} < \infty$  
(clearly, $\eta_{\star} \le 1$ for non-decreasing
$t \mapsto \pi^{(t)}(x)$).
\end{defn}

\begin{thm}\label{thm-main}
Suppose the walk is uniformly lazy, namely 
$\inf_{t,x} P(t,x;t+1,x) \ge \gamma$ for some 
$\gamma\in (0,1/2]$ and $\beta(u)$ of \eqref{eq:eff-non-dec} are
finite. Fixing $d>1$, we consider the isoperimetric growth function
\begin{align}
\psi_{d,\beta} (t) :=  \sum_{u=0}^{t-1} (\beta(u)^{1/d} \kappa_u)^2 \,,\quad
\kappa_{u} &:= \inf_{A \subset V_u, 0<|A|<\infty} \; \Big\{ \frac{\pi^{(u)}(A,A^c)}{\pi^{(u)}(A)^{(d-1)/d}} \Big\} \,,\label{eq:isop2} 
\end{align}
with $\psi_{d}(t)$ in case the factors $\beta(u)^{1/d}$ are omitted.
If for fixed $\lambda \in (0,1/2]$,  
\begin{equation}\label{eq:r-choice}
\exists \, r \in (s,t), \qquad 
\frac{\psi_{d,\beta} (r) -\psi_{d,\beta}(s)}{\psi_{d,\beta}(t)-\psi_{d,\beta}(s)}
 \in [\lambda,1-\lambda] \,,
\end{equation}
then for some $c_{+}=c_{+}(d,\gamma,\lambda)$ finite, 
any $t> s \ge 0$, $x \in V_s$ and $y \in V_t$
\begin{align}\label{unbdd-beta}
h(s,x;t,y) \le c_{+} \beta(t)(\psi_{d,\beta}(t)-\psi_{d,\beta} (s))^{-d/2}\,.
\end{align}
Let $\eta_0:=\sup_x \pi^{(0)}(x)$ (positive). For $t \mapsto \pi^{(t)}(x)$ 
effectively non-decreasing and
uniformly bounded (i.e. $C := \sup_{t,x} \pi^{(t)}(x) < \infty$), 
 we further have that for some
$c_{\star}=c_{\star}(d,\gamma,\eta_0,\eta_{\star},C)$ finite 
and all $s,x,t,y$ as above,
\begin{align}\label{eq:HK-ubd-imp}
{\pi^{(s)}(x)} h(s,x;t,y) \le c_{\star} \big(e+ \psi_d(t)-\psi_d(s) \big)^{- d/2}\,.
\end{align}
\end{thm}

\begin{remark}\label{rmk:trans} 
If the \abbr{rhs} of \eqref{eq:HK-ubd-imp} is summable over $t$, 
then $\sum_t P(0,x;t,y)$ is finite for any $x \in V_0$, $y \in V$.
Hence, the process $\{X_t\}$ is then \emph{transient in the strong sense} 
that starting 
at any non-random $X_0 \in V_0$ yields a finite expected number of visits 
to any $y \in V$ (and in particular, w.p.1. the sample path 
$t \mapsto X_t$ visits any $y \in V$ only finitely many times).
\end{remark}

\begin{remark}\label{rmk:conv-fast} 
Assuming $\kappa_u$ are bounded away from zero, even 
for polynomially growing $u \mapsto \beta(u)$ the 
\abbr{rhs} of \eqref{unbdd-beta} yields the optimal
$(t-s)^{-d/2}$ bound. For example, this applies when 
$\sup_{x} |\pi^{(t)}(x)/\pi(x) - 1| \to 0$ at rate 
$t^{-1}$. In contrast, for exponentially growing 
$u \mapsto \beta(u)$ the \abbr{rhs} of \eqref{unbdd-beta}
is $O(1)$, so carries no information. Indeed, the latter happens for 
the recurrent random walk among oscillating $[1-\ep,1+\ep]$-valued 
edge conductances on $\bZ^2\times\bZ_+$ which is given in 
\cite[Proposition 1.5(i)]{HK}.
\end{remark}

\begin{remark}\label{rem:cgz}
For $d > p \ge 1$ the $d$-dimensional 
Sobolev $\ell_p$-inequality holds on $\bG_u$, if 
\begin{align}
\wh{\kappa}_u:=\inf_{|\text{supp}(f)|<\infty}\Big\{\frac{\|\grad f\|_{p,u}}{\|f\|_{pd/(d-p),u}}\Big\}
\end{align}
is positive, with the corresponding functional norms for $q \ge 1$,
\begin{align*}
\|f\|_{q,u}&:=\big(\sum_{x\in V_u}|f(x)|^q\pi^{(u)}(x)\big)^{1/q},\\
\|\grad f\|_{q,u}&:=\big(\frac{1}{2}\sum_{x,y\in V_u}|f(y)-f(x)|^q\pi^{(u)}(x,y)\big)^{1/q}.
\end{align*}
Recall that for $d>1$, the Sobolev $\ell_1$-inequality is equivalent to the isoperimetric inequality of (\ref{eq:isop2}) with $\wh{\kappa}_u=\kappa_u$,
whereas for $d>2$, the Sobolev $\ell_2$-inequality is implied by the isoperimetric inequality (see \cite[Theorem 3.2.7]{Ku}).
\newline 
For uniformly lazy walk 
and \emph{time-independent conductances,} it is shown in \cite{CGZ}
that the Sobolev $\ell_2$-inequality 
with uniformly positive $\wh{\kappa}_u=\kappa$ yields the Gaussian heat kernel 
full upper bound (via the discrete integral maximum principle), and a 
matching on-diagonal lower bound holds under additional volume condition. 
\end{remark}

\begin{remark}\label{rmk:delayed}
In case of \emph{delayed random walk} one
specifies only $\{\pi^{(t)}(x,y), x \ne y\}$.
Then, assuming that for some $\gamma \in (0,1/2]$,
$$
\sup_{t,x} \pi^{(t)}(x,\{x\}^c) \le 1-\gamma \,,
$$ 
one lets $\pi^{(t)}(x,x) := 1 - \pi^{(t)}(x,\{x\}^c)$. It results 
with $\pi^{(t)}(x)=1$ for all $t,x$ and the uniformly lazy 
transition probabilities $P(t,x;t+1,y)=\pi^{(t)}(x,y)$ then
satisfy the heat-kernel upper bound \eqref{eq:HK-ubd-imp}.
\end{remark}

Here is a direct consequence of Theorem \ref{thm-main} 
(thanks to Remark \ref{rmk:trans}).
\begin{cor}\label{cor:u-ellip} 
Suppose $\bG$ of bounded degree satisfies a uniform 
isoperimetric inequality of order $d>2$
(e.g. the lattice $\bG=\bZ^d$), and consider
a uniformly lazy walk $\{X_t\}$ on $\bG$ equipped with 
\emph{uniformly elliptic} and bounded \emph{edge conductances}
(namely, $\pi^{(t)}(x,y)\in[C_1^{-1},C_1]$ for all $t$ and edges 
or self-loops $(x,y)$, with $C_1$ a universal finite constant).
\newline
If $t \mapsto \pi^{(t)}(x)$ are effectively non-decreasing, then
for any law of $X_0$ the expected number of visits by $\{X_t\}$ 
to $y \in V$ is finite (so w.p.1. the sample path 
visits each site finitely many times).
\end{cor}
\noindent
Indeed, in the setting of Corollary \ref{cor:u-ellip} we have
\eqref{eq:isop2} holding with $\kappa_u$ 
at least some universal positive constant 
times the edge-isoperimetic constant for $\bG$, hence uniformly bounded 
away from zero. This yields the linear growth of $\psi_d(\cdot)$ with 
$P(s,x;t,y) \le c_\star (t-s)^{-d/2}$, hence the stated 
strong transience (when $d>2$), uniformly in $X_0$. 

\smallskip
We note in passing that having only $\pi^{(t)}(x)\in[C^{-1},C]$ for all $x \in V$,
is not enough (for example the graph $\bZ^d$ without all edges connecting 
finite box $\bB_r$ to $\bB_{r}^c$ has uniformly bounded vertex 
conductances, but $\kappa_u=0$ in \eqref{eq:isop2} and starting 
at $X_0=0$ any random walk on this graph is confined to $\bB_{r}$, 
hence recurrent). 

\smallskip
The analog of Corollary \ref{cor:u-ellip} applies also
for the continuous time, constant speed random walk, the 
definition of which we provide next.

\begin{defn}\label{defn:csrw}
Suppose $\bG=(V,E)$ is locally finite graph 
equipped with \abbr{rcll} edge conductances 
$t \mapsto \pi^{(t)}(x,y)$ such that $\pi^{(t)}(x)>0$ for all $x$.
The $V$-valued stochastic process $\{Y_t\}$ of \abbr{RCLL} sample path 
$t \mapsto Y_t$ is called  
a constant speed random walk (in short \abbr{CSRW}), if it waits 
i.i.d. $\exp(1)$ times between successive jumps, and if 
$Y_{T^-} = x$ just prior to the current random jump time $T$, then 
the process jumps across each $(x,y) \in E$ with probability 
$\pi^{(T)}(x,y)/\pi^{(T)}(x)$.
\end{defn}

{\begin{defn}\label{rcll-eff-non-dec}
We call \abbr{RCLL} vertex conductances $t\mapsto\pi^{(t)}(x)$  
effectively non-decreasing, if for Lebesgue a.e. $t_k \uparrow \infty$, 
the sequence $k\mapsto\pi^{(t_k)}(x)$ is effectively non-decreasing 
(see Definition \ref{eff-non-dec}).
\end{defn}}

\begin{ppn}\label{csrw}
Suppose graph $\bG=(V,E)$ of bounded degree that satisfies 
a uniform isoperimetric inequality of order $d>2$ (e.g. the lattice $\bG=\bZ^d$), is 
equipped with \emph{uniformly elliptic} and bounded \abbr{RCLL}
\emph{edge conductances}
(namely, $\pi^{(t)}(x,y)\in[C_1^{-1},C_1]$ for all $t \ge 0$ and 
$(x,y) \in E$, with $C_1$ some universal finite constant). 
Assuming further that $t \mapsto \pi^{(t)}(x)$ are effectively non-decreasing, w.p.1. the sample path $t \mapsto Y_t$ of the \abbr{CSRW} returns to 
any $y \in V$ only finitely many times.
\end{ppn}

\medskip
In many non-elliptic settings we get fast enough isoperimetric growth  
for \eqref{eq:HK-ubd-imp} to yield the desired a.s. transience. Even 
when it does not, such result may be obtained by taking advantage of 
a-priori bounds on the support of the relevant evolving set. We next
deal with one such example, which partially resolves the open question 
raised in \cite[Remark 1.12]{DHS}.

\begin{ppn}\label{percol}
Let $\bD_0$ denote the unique infinite cluster of  {
the correlated percolation model of \cite[Theorem 1.2]{PRS} (which includes as special case the
Bernoulli($p$) bond percolation 
at super-critical $p > p_c(\bZ^d)$),  on $\bZ^d$, $d>2$,} conditioned 
to contain the origin. Starting with $X_0$ at the origin, 
the sample path of any uniformly lazy
\abbr{srw} on growing connected sub-graphs $\{\bD_t\}$ of 
the lattice $\bZ^d$ sharing the vertex set $\mathcal{V}(\bD_0)$ of 
$\bD_0$ (with uniformly bounded self-loops, hence vertex, conductances),
is strongly transient in the sense of Remark \ref{rmk:trans}.
\end{ppn}

As mentioned before, our key tool is the evolving set process $\{S_t\}$,
where $S_t$ is the following random finite subset of $V_t$, $t \ge 0$. 
\begin{defn}\label{defn-evol}
Starting with $S_0=\{x\}$ for $x \in V_0$, sequentially 
for $t=0,1,2,\ldots$ we let $U_{t+1}$ denote a 
Uniform(0,1) random variable which is 
independent of $\{S_s,X_s,U_s,  0 \le s \le t\}$, and form
\begin{align*}
S_{t+1}=\{y \in V_{t+1} : \frac{\pi^{(t)}(S_t,y)}{\pi^{(t+1)}(y)}\ge U_{t+1}\}.
\end{align*}
Assuming $t\to\pi^{(t)}(x)$ are non-decreasing, it 
follows that $V_t \subseteq V_{t+1}$ and for every $y\in V_{t+1}$
\begin{align}
\mathbb{P}(y\in S_{t+1}|S_t)=\frac{\pi^{(t)}(S_t,y)}{\pi^{(t+1)}(y)} \label{update}
\end{align}
(the \abbr{RHS} of \eqref{update} is well defined $[0,1]$-valued 
and any $y$ accessible from $S_t$ must be in $V_{t+1}$).
\end{defn}

\begin{remark}\label{rem:comp}
For uniformly lazy random walk having $\pi^{(t)}(x)$ independent of $t$ 
(so w.l.o.g. $V_t=V$ for all $t$), 
one has the analogue of \cite[Lemma 8]{MP2}. That is, if $(S_t)$ is 
an evolving set process, then the sequence $(S_t^c)$ is also an evolving set 
process of the same transition probability. The proof in \cite[pg 253]{MP2} 
can be reproduced using $\pi^{(t+1)}(x)=\pi^{(t)}(x)$ for all $t,x$, and 
noting that for Uniform$(0,1)$ random variable $U \stackrel{(d)}{=} 1-U$.
\end{remark}

\medskip
We further utilize the concept of conditioned (or size-biased) evolving set, 
upon adapting it to our \emph{parabolic} time-dependent setting.
In particular, it yields the following extension of 
\cite[Theorem 17.23]{LPW} originally due to \cite{DF}.
\begin{defn}\label{defn-DF-evol}
We say that $(S_t \subseteq V_t)$ is the conditioned evolving set, starting 
at $S_0=\{x\}$, if it has the transition kernel
\begin{align}\label{eq:whK-def}
\widehat{K}(t,A;t+1,B)=\frac{\pi^{(t+1)}(B)}{\pi^{(t)}(A)}K(t,A;t+1,B),
\end{align}
where $K(\cdot;\cdot)$ is the transition kernel of 
the unconditioned evolving set of Definition \ref{defn-evol}. 
\end{defn}

\begin{ppn}\label{prop-df}
Suppose $t\mapsto\pi^{(t)}(x)$ are non-decreasing and $(X_t,S_t)$ starting from $(X_0,S_0)=(x,\{x\})$ follows the  
time-varying Markov transition kernel $P^*$ on $V \times 2^{V}$, 
given for $x \in A \cap V_t$, $\pi^{(t)}(x,y) >0$, by
\begin{align*}
P^*(t,(x,A);t+1,(y,B))&=P(t,x;t+1,y)\mathbb{P}(S_{t+1}=B|y\in S_{t+1},S_t=A)\mathbb{I}_{\{y\in B\}}\\
&=\frac{P(t,x;t+1,y)K(t,A;t+1,B)\pi^{(t+1)}(y)\mathbb{I}_{\{y\in B\}}}{\pi^{(t)}(A,y)}.
\end{align*}
\begin{itemize}
\item[(a)]
The marginal process $t \mapsto X_t$ is a time in-homogeneous Markov 
process having the transition kernel $P$, and the marginal process
$t \mapsto S_t$ is another time in-homogenous Markov chain whose 
transition kernel is $\widehat{K}(\cdot,\cdot)$ of \eqref{eq:whK-def}.
\item[(b)] For any $t$, $x \in V_0$ and $w\in S_t$,  
\begin{align*}
\mathbb{P}^*_{x,\{x\}}(X_t=w|S_0,...,S_t)=\frac{\pi^{(t)}(w)}{\pi^{(t)}(S_t)}\,.
\end{align*}
\end{itemize}
\end{ppn}

We next list a few open problems.
\begin{pbm} For time-independent conductances
\cite{CGZ} relies,
in the setting of Remark \ref{rem:cgz},
on using the time-reversed chain.
\newline
(a). Can this idea be extended to monotone and 
genuinely time varying path of 
reversing measures $t \mapsto \{\pi^{(t)}(x), x \in V\}$?
\newline
(b). Alternatively, does the bound \eqref{eq:HK-ubd-imp} hold 
for uniformly elliptic, uniformly lazy and bounded
edge conductances for which $t \mapsto \pi^{(t)}(x)$ are strictly 
monotone decreasing in $t$?
\newline
(c). Is it possible to establish for monotone increasing 
reversing measures a Gaussian type off-diagonal upper bound and 
somewhat comparable lower bounds?
\end{pbm}

\begin{pbm} Extend Proposition \ref{percol} to allow adding new vertices as $\bD_t$ evolves. 
\newline
(a). For example, start with $\bD_0$ the 
unique infinite cluster of super-critical Bernoulli bond percolation on $\bZ^d$, $d>2$
and end with the full lattice $\bD_\infty=\bZ^d$.
\newline
(b). Alternatively, consider finite graphs $\{\bD_t\}$ that grow to a 
transient infinite graph $\bD_\infty$ of uniformly bounded degrees. 
Slow growth can yield recurrence of the walk, with a sharp phase transition 
from recurrence to transience 
in terms of the growth rate 
predicted for $\bD_\infty=\bZ^d$, $d>2$ (see \cite[Theorem 1.4, Conjecture 1.2]{DHS}). 
Extend the scope of evolving sets to resolve this prediction.
\end{pbm}

Section \ref{sec-pf-main} is devoted to 
the proof of Theorem \ref{thm-main} 
which partly builds on \cite{MP2} 
(and at places also on \cite[Ch. 3]{Ku}), while Propositions 
\ref{csrw}, \ref{percol} and \ref{prop-df} 
are proved in Section \ref{sec-percol}. 
\end{section}

\begin{section}{Proof of Theorem \ref{thm-main}}\label{sec-pf-main}

We start with two key facts about the evolving set 
process of Definition \ref{defn-evol}, in case
$t\mapsto\pi^{(t)}(x)$ are non-decreasing.
\begin{lem} \label{heat kernel}
The sequence $\{\pi^{(t)}(S_t)\}$ is a martingale and 
for any $t \ge 0$, $x \in V_0$ and $y \in V$
\begin{align}\label{eq:trans-prob}
P(0,x;t,y)=\frac{\pi^{(t)}(y)}{\pi^{(0)}(x)}\mathbb{P}_{\{x\}}(y\in S_t).
\end{align}
\end{lem}
\begin{proof} Fixing hereafter the starting state $S_0=\{x\}$ in $V_0$, 
we have from \eqref{update} that,
\begin{align*}
&\mathbb{E} (\pi^{(t+1)}(S_{t+1})|S_t) =\mathbb{E}\big[
\sum_{z \in V_{t+1}} \mathbb{I}_{\{z\in S_{t+1}\}}\pi^{(t+1)}(z)|S_t\big]\\
&=\sum_{z \in V_{t+1}} \mathbb{P}(z\in S_{t+1}|S_t)\pi^{(t+1)}(z)=
\sum_{z \in V_{t+1}} 
\frac{\pi^{(t)}(S_t,z)}{\pi^{(t+1)}(z)}\pi^{(t+1)}(z)=\pi^{(t)}(S_t)\,.
\end{align*}
That is, $\{\pi^{(t)}(S_t)\}$ is a martingale. 

Turning to confirm the identity \eqref{eq:trans-prob}, note 
first that when $t=0$, both sides of it equal $\mathbb{I}_{\{y=x\}}$. 
Next, if this identity holds for $t$, then using Chapman-Kolmogorov, our 
induction hypothesis, the formula for $P(t,z;t+1,y)$ and \eqref{update}, 
we find that
\begin{align*}
P(0,x&;t+1,y)=\sum_{z \in V_{t}}P(0,x;t,z)P(t,z;t+1,y)\\
&=\sum_{z \in V_t} \frac{\pi^{(t)}(z)}{\pi^{(0)}(x)}\mathbb{P}_{\{x\}}(z\in S_t)P(t,z;t+1,y)\\
&=\frac{1}{\pi^{(0)}(x)}\mathbb{E}_{\{x\}}\Big[ \sum_{z \in S_t} 
\pi^{(t)}(z)P(t,z;t+1,y)\Big] 
=\frac{1}{\pi^{(0)}(x)}\mathbb{E}_{\{x\}}\big[\pi^{(t)}(S_t,y)\big]\\
&=\frac{1}{\pi^{(0)}(x)}
\mathbb{E}_{\{x\}}\Big[\pi^{(t+1)}(y)\mathbb{P}(y\in S_{t+1}|S_t)\Big]
=\frac{\pi^{(t+1)}(y)}{\pi^{(0)}(x)} \mathbb{P}_{\{x\}}(y\in S_{t+1})\,.
\end{align*}
Thus, by induction \eqref{eq:trans-prob} holds for all $t$.
\end{proof}

The next result is essential to our proof 
and the only place 
where we utilize the assumed isoperimetric inequality \eqref{eq:isop2}.
\begin{lem}\label{recursion}
For some $\wt{c}=\wt{c}(\gamma)$ positive, $\beta:=\alpha-2/d$,
any $\alpha\in(0,1)$, $t \ge 0$ and $x \in V_0$, 
\begin{align}\label{eq:recur}
\mathbb{E}_{\{x\}}\Big[\pi^{(t+1)}(S_{t+1})^\alpha-\pi^{(t)}(S_t)^\alpha|S_t
\Big]\le -\wt{c} \alpha (1-\alpha) \kappa_t^2 \pi^{(t)}(S_t)^{\beta
}\, \mathbb{I}_{\{\pi^{(t)}(S_t)>0\}}.
\end{align}
Further, for $\alpha>1$ we have the converse bound 
\begin{align}\label{eq:recur-2}
\mathbb{E}_{\{x\}} \Big[\pi^{(t+1)}(S_{t+1})^\alpha-\pi^{(t)}(S_t)^\alpha|S_t
\Big]\ge \wt{c} \alpha (\alpha-1) \kappa_t^2 \pi^{(t)}(S_t)^{\beta
}\, \mathbb{I}_{\{\pi^{(t)}(S_t)>0\}}.
\end{align}
\end{lem}
\begin{proof} Note that $\pi^{(t)}(S_t) = 0$ iff $S_t = \emptyset$, in which 
case by Definition \ref{defn-evol} also $S_{t+1} = \emptyset$ and 
our claim trivially holds. Assuming hereafter that $\pi^{(t)}(S_t)>0$, 
since $U_{t+1}$ is independent of $S_t$
we deduce from (\ref{update}) 
that for every $y \in V_{t+1}$
\begin{align}\notag
p_\star(y,t) : &=
\mathbb{P}\big(y\in S_{t+1} \big| U_{t+1}\le1/2, S_t\big) \\
&=\mathbb{P}\Big(
U_{t+1}\le \frac{\pi^{(t)}(S_t, y)}{\pi^{(t+1)}(y)}\Big| U_{t+1}\le1/2, S_t\Big) 
= 1\wedge \frac{2\pi^{(t)}(S_t, y)}{\pi^{(t+1)}(y)}\,.
\label{eq:pstar-def}
\end{align}
Next, let 
\begin{align}
\Delta_t &:= \frac{1}{\pi^{(t)}(S_t)} 
\sum_{y \in V_{t+1}} \pi^{(t+1)}(y) p_\star(y,t) \nonumber \\
&=
\frac{1}{\pi^{(t)}(S_t)} 
\sum_{y \in V_{t+1}}\big[\pi^{(t+1)}(y)\wedge 2\pi^{(t)}(S_t,y)\big]
\,.
\label{eq:del-dfn}
\end{align}
By assumption, our lazy random walk is such that 
$\pi^{(t)}(y,y) \ge \gamma \pi^{(t)}(y)$ for some $\gamma \in (0,1/2)$. 
Consequently, for any $y\in S_t$, 
\beq\label{eq:basic-bd}
\pi^{(t)}
(S_t,y) \ge \pi^{(t)}(y,y) \ge 
\gamma\pi^{(t)}(y)
\ge\frac{\gamma}{1-\gamma}\pi^{(t)}(S^c_t,y) \,.
\eeq
Now, since $t \mapsto \pi^{(t)}(y)$ is non-decreasing, it follows from 
\eqref{eq:pstar-def} and \eqref{eq:basic-bd} that for $y \in S_t$,
\begin{align*}
&\pi^{(t+1)}(y) p_\star(y,t) 
=\pi^{(t+1)}(y)\wedge2\pi^{(t)}(S_t,y) \ge\pi^{(t)}(y)\wedge2\pi^{(t)}(S_t,y)\\
&=\pi^{(t)}(S_t,y)+\pi^{(t)}(S_t^c,y)\wedge\pi^{(t)}(S_t,y)
\ge \pi^{(t)}(S_t,y)+\frac{\gamma}{1-\gamma}\pi^{(t)}(S_t^c,y)\,.
\end{align*}
Likewise, for $y\in S_t^c$,
\begin{align*}
\pi^{(t+1)}(y) p_\star(y,t) 
\ge\pi^{(t)}(S_t,y)+\frac{\gamma}{1-\gamma}\pi^{(t)}(S_t,y)\,.
\end{align*}
Letting 
$$
R_t:=\frac{\pi^{(t)}(S_t, S_t^c)}{\pi^{(t)}(S_t)} \,,\quad 
\quad 
\Gamma_t:=\frac{\pi^{(t+1)}(S_{t+1})}{\pi^{(t)}(S_t)}
$$ 
we find upon combining the preceding inequalities 
with the definition \eqref{eq:del-dfn} of $\Delta_t$, that
\begin{align}\label{eq:del-ineq}
\Delta_t \ge \frac{1}{\pi^{(t)}(S_t)} [\pi^{(t)}(S_t)+
\frac{2\gamma}{1-\gamma}\pi^{(t)}(S_t,S_t^c)] = 
1+\frac{2\gamma}{1-\gamma}R_t\,.
\end{align}
Further, with $\pi^{(t)}(S_t)$ a martingale and 
$U_{t+1}$ independent of $S_t$, we have that 
\begin{align*}
1 = \mathbb{E}\big( \Gamma_t \big| S_t \big)
= \frac{1}{2} \mathbb{E} \big(\Gamma_t \big| U_{t+1}\le 1/2, S_t\big) + 
\frac{1}{2} \mathbb{E}\big( \Gamma_t \big|U_{t+1}> 1/2,S_t \big) \,.
\end{align*}
But, from the definition of $\Delta_t$ and of $p_\star(y,t)$ we deduce that 
\begin{align*}
 \mathbb{E} \big(\Gamma_t \big| U_{t+1}\le 1/2, S_t\big) = \Delta_t \,, \quad
\quad \mathbb{E}\big(\Gamma_t \big|U_{t+1}> 1/2,S_t \big)=2-\Delta_t \,.
\end{align*}
Considering first {$\alpha \in (0,1)$},
by Jensen's inequality and the preceding identities,
\begin{align}
\mathbb{E}(\Gamma_t^\alpha|S_t)&=\frac{1}{2} 
\mathbb{E}(\Gamma_t^\alpha|U_{t+1}\le 1/2, S_t)+
\frac{1}{2} \mathbb{E}(\Gamma_t^\alpha|U_{t+1}> 1/2, S_t)\nonumber \\
&\le \frac{1}{2} 
\Big[\mathbb{E}(\Gamma_t |U_{t+1}\le 1/2, S_t)\Big]^\alpha+
\frac{1}{2} \Big[\mathbb{E}(\Gamma_t|U_{t+1}> 1/2, S_t)\Big]^\alpha \nonumber \\
&= \frac{1}{2} \Delta_t^\alpha + \frac{1}{2} (2-\Delta_t)^\alpha
=: f_\alpha (\Delta_t-1) \,. \label{eq:f-alpha-bd}
\end{align}
Next note {that the even function $f_\alpha(\cdot)$ is non-increasing 
on $[0,1]$ when $\alpha \in (0,1)$ and non-decreasing on $[0,1]$
for any other $\alpha \in \R$. Further, $f_\alpha(0)=1$ and
$f_\alpha''(y) = \alpha (\alpha-1) f_{\alpha-2}(y)$. Hence, for $y \in [0,1]$, 
\begin{align}\label{eq:f-al-ubd}
f_\alpha(y) &\le 1 +  \alpha (\alpha-1) \frac{y^2}{2} \,,\qquad \,\alpha \in (0,1), \\
f_\alpha(y) &\ge 1 + \alpha (\alpha-1) \frac{y^2}{8} \,,\qquad \,\alpha \ge 1 \,.
\label{eq:f-al-lbd}
\end{align}
}
It thus follows from  \eqref{eq:del-ineq}--\eqref{eq:f-al-ubd} 
that when $\alpha \in (0,1)$, 
\begin{align}\label{eq:Gamma-ineq}
\mathbb{E}(\Gamma_t^\alpha | S_t) & \le f_\alpha(\Delta_t-1) \le f_\alpha
\Big(\frac{2\gamma}{1-\gamma} R_t\Big)
\le 1-\frac{2\alpha(1-\alpha)\gamma^2}{(1-\gamma)^2}R_t^2 \,.
\end{align}
Our assumption that $\mathbb{G}$ is locally finite, and the construction of
the evolving set $\{S_t\}$ guarantees the finiteness of each $S_t$. Hence, 
from \eqref{eq:isop2} we have that for any $t \ge 0$,
\begin{equation}\label{eq:kappa-enters}
R_t \ge \kappa_t \pi^{(t)}(S_t)^{-1/d} \,.
\end{equation}
Thus, from \eqref{eq:Gamma-ineq} we conclude that for some 
positive $\wt{c}=\wt{c}(\gamma)$ and all $t$, 
\begin{align}
\mathbb{E}\Big[ \frac{\pi^{(t+1)}(S_{t+1})^\alpha}{\pi^{(t)}(S_t)^\alpha}-1
\big| S_t\Big] 
& \le -\frac{2\alpha(1-\alpha)\gamma^2}{(1-\gamma)^2} R_t^2 
\nonumber \\
&\le -\wt{c} \alpha (1-\alpha) \kappa_t^2 \pi^{(t)}(S_t)^{-2/d} \,,
\label{eq:alpha-bd}
\end{align}
and multiplying both sides by $\pi^{(t)}(S_t)^\alpha$ yields the upper bound 
of \eqref{eq:recur}. 

Turning to the proof of \eqref{eq:recur-2}, similarly 
to the derivation of \eqref{eq:f-alpha-bd} and \eqref{eq:Gamma-ineq}
we get from \eqref{eq:del-ineq} and \eqref{eq:f-al-lbd}
that when $\alpha>1$,
\begin{align*}
\bE(\Gamma_t^\alpha|S_t)\ge 
f_\alpha(\Delta_t-1) \ge 
 f_\alpha\big(\frac{2\gamma}{1-\gamma}R_t\big)
\ge
 1+\frac{\alpha(\alpha-1)\gamma^2}{2(1-\gamma)^2}R_t^2\,.
\end{align*}
Using (\ref{eq:kappa-enters}) we find, similarly to 
 the derivation of 
\eqref{eq:alpha-bd},
that now,
\begin{align}\label{eq:alpha-bd-2}
\mathbb{E}\Big[ \frac{\pi^{(t+1)}(S_{t+1})^\alpha}{\pi^{(t)}(S_t)^\alpha}-1
\big| S_t\Big] 
\ge \wt{c} \alpha (\alpha-1) \kappa_t^2 \pi^{(t)}(S_t)^{-2/d} \,,
\end{align}
ending with (\ref{eq:recur-2}).
\end{proof}

Our next lemma embeds 
$\{\pi^{(t)}(S_t)\}$ as the integer time samples of 
a continuous martingale (assuming as before that
$t\mapsto\pi^{(t)}(x)$ are non-decreasing).
\begin{lem}\label{embedding}
There exists a martingale $(M_u, u \ge 0)$ of a.s. continuous sample path,
such that $M_i=\pi^{(i)}(S_i)$ for $i\in\N$ {and 
$\tau=\inf\{u \ge 0 : M_u \le 0 \}$ is 
$\N \cup \{\infty\}$-valued.}
\end{lem}
\begin{proof} With
$\Phi(\cdot)$ the standard normal \abbr{cdf} and
$(B_s, s \ge 0)$ a
standard Brownian motion,   
let $S_0=\{x\}$ and
$U_{i+1}=\Phi(B_{i+1}-B_i)$ the i.i.d. 
Uniform$(0,1)$ variables used to construct $S_{i+1}$ 
from $S_i$ in Definition \ref{defn-evol}. The process $\{S_i\}$ is 
then adapted to 
$\cF_u:=\sigma\{B_s,  s \in [0,u] \}$.
Considering the 
$\cF_u$-adapted process 
\begin{align}\label{eq:cont-mg}
M_u:=\bE[\pi^{(i+1)}(S_{i+1})|\cF_u] \,, \qquad  \forall u \in [i,i+1) \,, i \in \N\,,
\end{align}
we have by the independence of Brownian increments and Lemma \ref{heat kernel}, 
that for any $i \in \N$,
\begin{align}\label{eq:disc-mg}
M_i = \bE[\pi^{(i+1)}(S_{i+1})| \cF_i] = \bE[\pi^{(i+1)}(S_{i+1})|S_i]=\pi^{(i)}(S_i) \,.
\end{align}
Clearly, $(M_u,\cF_u)$ is a (Doob) martingale within each interval $[i,i+1)$.
Upon plugging \eqref{eq:disc-mg} at $i+1$ within
\eqref{eq:cont-mg}, the martingale property
extends to $[i,i+1]$, which by the law of iterated expectations yields 
that $(M_u,\cF_u)$ is a martingale for all $u \ge 0$. 
Turning to the continuity of $u \mapsto M_u$, for any $i \in \N$, $y \in V_{i+1}$ 
and $A \subseteq V_{i}$ let
\begin{align*}
H_i(A,y):=\Phi^{-1}\Big(\frac{\pi^{(i)}(A,y)}{\pi^{(i+1)}(y)}\Big)\,.
\end{align*}
By Definition \ref{defn-evol} and the independence of Brownian increments, we have that 
for any $s\in [0,1)$ and $i \in \N$, 
\begin{align}
M_{i+s}
&=\sum_{y\in V_{i+1}}\pi^{(i+1)}(y)\bP\big(H_i(S_i,y)\ge B_{i+1}-B_i \,| S_i, B_{i+s}-B_i \big) \nonumber \\
&=\sum_{y\in V_{i+1}}\pi^{(i+1)}(y)\Phi\Big(\frac{H_i(S_i,y)-B_{i+s}+B_i}{\sqrt{1-s}}\Big). \label{continuity}
\end{align}
With $s \mapsto B_{i+s}$ continuous, each term of the sum on the \abbr{rhs} of \eqref{continuity} is continuous in $s \in [0,1)$. Having
$\G$ locally finite, only finitely many 
$y \in V$ for which $H_i(S_i,y) \ne -\infty$ contribute to 
that sum, hence $u \mapsto M_u$ is continuous on $[i,i+1)$. 
Further, a.s. $H_i(S_i,y) \ne B_{i+1}-B_i$ for all $y \in V_{i+1}$, in which case
by the continuity of $u \mapsto B_u$ at $i+1$,  
\begin{align*}
\lim_{s\uparrow 1} \Phi\Big(\frac{H_i(S_i,y) - B_{i+s} + B_i}{\sqrt{1-s}}\Big)  
= \bI\{H_i(S_i,y) \ge B_{i+1} - B_i \} \,.
\end{align*}
Upon comparing \eqref{continuity} with Definition \ref{defn-evol}, this extends
the continuity of $u \mapsto M_u$ to $[i,i+1]$ and thereby to all $u \ge 0$.
\newline
{Finally, $M_u$ is non-negative by \eqref{eq:cont-mg},
whereas by \eqref{continuity} it is strictly positive on $[i,i+1)$ unless 
$H_i(S_i,y) = -\infty$ for all $y$, namely $S_i = \emptyset$} 
(in which case $M_u=0$ for all $u \ge i$).
\end{proof}

\medskip
\noindent
{\emph{Proof of Theorem \ref{thm-main}.}} It suffices to prove
\eqref{unbdd-beta} and \eqref{eq:HK-ubd-imp} for $s=0$, as  
$s \in (0,t)$ then follows by considering the edge conductances 
$\{\pi^{(s+\cdot)}\}$ starting at $X_s = x \in V_s$ (and
consequently, using $\beta(u)/\beta(s)$ and
$\psi_{d,\beta}(t)-\psi_{d,\beta}(s)$ instead of 
$\beta(u)$ and $\psi_{d,\beta}(t)$). 

Fixing hereafter $s=0$, 
we start with a short derivation of the sub-optimal bound 
$P(0,x;t,y) \le C' \psi_d(t)^{-(1-\alpha) d/2}$ for 
$\alpha \in (0,1)$, 
non-decreasing $t \mapsto \pi^{(t)}(y) \le C$,
and some $C'=C'(d,\alpha,\gamma,C)$ finite. Indeed,  
\eqref{eq:trans-prob} then 
result with $P(0,x;t,y)\le C^{1-\alpha} m_t$ for 
$m_t = \bE_{\{x\}} [M_t^\alpha]/M_0$ and $M_t=\pi^{(t)}(S_t)$. 
Further, with $\beta=\alpha - \delta (1-\alpha)$, the elementary bound 
\beq
\label{eq:elem-jensen}
\bE [Z^{\beta} \mathbb{I}_{Z>0} ] \ge (\bE [Z^\alpha])^{1+\delta} \,,
\eeq
holds for $Z=M_t/M_0 \ge 0$ of mean one and $\delta>0$. Taking 
the expectation of \eqref{eq:recur}, it thus follows from \eqref{eq:elem-jensen} 
that for $\delta=2/((1-\alpha)d)$,
\beq
m_{t+1} \le m_t \exp(-\widetilde{c} \alpha (1-\alpha) \kappa_t^2 m_t^\delta) \,,
\eeq
and consequently $m_t \le c' \psi_d(t)^{-1/\delta}$ for 
some $c'(\alpha,d,\gamma)$ finite, as claimed. 

However, the sharp bound 
\eqref{eq:HK-ubd-imp} (where $\alpha=0$), requires the more 
elaborate argument provided next,
where we first derive \eqref{eq:HK-ubd-imp} out of \eqref{unbdd-beta}
in case $\pi^{(u)}(x)$ are effectively non-decreasing and uniformly 
bounded. Indeed, by its definition in \eqref{eq:isop2},
\begin{equation}\label{eq:kappa-bd}
\kappa_i \le \inf_{v \in V_i} \pi^{(i)}(\{v\})^{1/d} \le C^{1/d} \,, \qquad \forall i \ge 0 
\end{equation}
and consequently $\beta(u)^{1/d} \kappa_u \le (\eta_{\star} C)^{1/d}$.
Thus, condition \eqref{eq:r-choice} holds (for $\lambda=1/3$) whenever
$\psi_{d,\beta} (t) \ge 3 (\eta_{\star} C)^{2/d}$. Since $\pi^{(t)}(x) \le C$,
it follows from \eqref{eq:eff-non-dec} that 
$$
\beta(t) \ge \sup_x \Big\{\frac{\pi^{(0)}(x)}{\pi^{(t)}(x)}\Big\} 
\ge \frac{\eta_0}{C}\,,
$$  
hence the condition \eqref{eq:r-choice} holds whenever 
\begin{equation}\label{eq:xi-bd}
\xi(t) := (\eta_{\star}/\beta(t))^{2/d} \psi_{d,\beta}(t) \ge 
3 (\eta_{\star}^2 C^2/\eta_0)^{2/d} \,,
\end{equation}
in which case
multiplying the inequality \eqref{unbdd-beta} by $\pi^{(0)}(x)$ yields 
the bound
\begin{align}\label{eq:HK-ubd}
\pi^{(0)} (x) h(0,x;t,y) \le c_{+} C \eta_{\star} \xi(t)^{-d/2} 
\le c_{\star} (e + \xi(t) )^{-d/2} \,,
\end{align}
for some $c_\star=c_\star
(d,c_+,\eta_0,\eta_\star,C)$ finite.
Next, recall that by \eqref{eq:ht-def} and \eqref{eq:trans-prob},  
for any $t \in \N$,
\begin{align}
\pi^{(0)} (x) h(0,x;t,y)= \bP_{\{x\}}(y\in S_t)
\le \bP_{\{x\}}(S_t \ne \emptyset) 
= \bP_{\{x\}}(M_t \ne 0) \,,
\label{non-empty}
\end{align}
for the continuous, non-negative  
$\bP$-martingale $\{M_u\}_{u\ge 0}$ of Lemma \ref{embedding}.
In view of \eqref{non-empty}, the \abbr{lhs} of \eqref{eq:HK-ubd} 
is at most one, hence increasing $c_\star$ guarantees that 
\eqref{eq:HK-ubd} trivially holds whenever \eqref{eq:xi-bd} fails.
Having effectively non-decreasing $t \mapsto \pi^{(t)}(x)$, implies
further that $\xi(t) \ge \psi_d(t)$ and thus \eqref{eq:HK-ubd-imp}
is a consequence of \eqref{eq:HK-ubd}.

Turning to the proof of \eqref{unbdd-beta}, note that  
multiplying all edge conductances $\{\pi^{(u)}(x,y)\}$ by a common 
factor does not effect the transition probabilities of 
the associated random walk at step $u$. Hence, re-defining 
the edge conductances 
$$
\wh{\pi}^{(u)}(x,y)=
\beta(u)
\pi^{(u)}(x,y), \quad u\in\N, \quad (x,y)\in E,
$$
results with $h(s,x;t,y)=\beta(t)\wh{h} (s,x;t,y)$,   
$\psi_{d,\beta}(\cdot)=\wh{\psi}_d(\cdot)$ and non-decreasing 
$u\mapsto \wh{\pi}^{(u)}(x)$. We consequently proceed to bound
the \abbr{rhs} of \eqref{non-empty}, for non-decreasing
$u\mapsto\pi^{(u)}(x)$ and $\beta(u) \equiv 1$. To this end, 
we utilize the stopping times 
\begin{align}\label{def:st-tm}
\tau_k:={\inf\{u\ge 0: M_u\ge e^k\}},
\quad T_k':={\inf\{i \in \N \cap (\tau_{k},\infty) 
: M_i = 0 \}} 
\end{align}
and note that for $r \in (0,t)$ of \eqref{eq:r-choice} 
and any $k \in \Z$, 
\begin{equation}\label{eq:stop-bd}
{\{M_t \ne 0 \}\subseteq} \{\tau_k>r\}\cup\{\tau_k \le r,\, T'_k>t\} \,.
\end{equation}
Further, for 
$\wt{M}:=\sup_{u\ge0} \{M_u\}$ and 
$E_k:=\{e^k\le \wt{M}<e^{ k+1}\}$, by Doob's inequality 
\begin{equation}\label{eq:doob-bd}
\bP_{\{x\}} (E_k) \le \bP_{\{x\}}{(\wt{M} \ge e^k)} \le \pi^{(0)}(x) e^{-k} \,.
\end{equation}
Thus, fixing $\vep \in (0,1)$ and  
setting $k_0:=\lfloor \log \pi^{(0)}(x) \rfloor$,
$L:=\lceil \log (\vep^2 \psi_d(t)^{d/2}) \rceil$, 
we get from 
\eqref{eq:stop-bd} and
\eqref{eq:doob-bd} that  
\begin{align}
\bP_{\{x\}} (M_t \ne 0)  \le
\bP_{\{x\}}{ (\wt{M}\ge e^L)} &+
\sum_{k=k_0}^{L-1}
\bP_{\{x\}}({\{M_t \ne 0\}} \cap E_k)
\nonumber\\
 \le \pi^{(0)}(x) \big[ e^{-L} &+ 
\sum_{k=k_0}^{L-1} e^{-k} \bP_{\{x\}}(\tau_k>r|{\wt{M}\ge e^k}) 
\nonumber\\
&+
\sum_{k=k_0}^{L-1} e^{-k} \bP_{\{x\}}(T'_k > t | E_k,\tau_k\le r) \big] \,. 
\label{reduce-to-tail}
\end{align}

Noting that $e^{-L}$ is of $O(\psi_d(t)^{-d/2})$ 
size, the remainder of the proof consists of three 
steps. First, by the continuity of our non-negative 
martingale, and the lower bound of \eqref{eq:recur-2}
on its quadratic variation, we show in {\bf Step I} that 
conditioning on $\{\wt{M} \ge e^k\}$ transforms the law 
of $\{S_0,\ldots,S_r\}$ to that of 
Definition \ref{defn-DF-evol}. Then, {\bf Step II}
shows that the probability of 
$\max_{i \le r} \{\pi^{(i)}(S_i)\}$ not exceeding 
$e^k$ for such size-biased evolving sets, is 
at most $O(\exp(-c \psi_d(r) e^{-2k/d}))$ and 
as a result the left sum in \eqref{reduce-to-tail}
is at most $O(\psi_d(r)^{-d/2})$ (see \eqref{bound-1}).
Noting that under $\{\tau_k \le r\}$ the probability 
of $E_k = \{\tau_{k+1}=\infty\}$ is  
bounded away from zero, {\bf Step III} 
controls the right sum over $k$ 
in \eqref{reduce-to-tail},
as $\{E_k, \tau_k \le r\}$ dictates a downward 
path $e^{a_{i+1}}$ driving $u \mapsto \pi^{(u)}(S_u)$, $u=\lceil \tau_k \rceil + i$, to zero at $u=t$, 
or else the super-martingale $Q_{i \wedge \sigma} \ge 0$ 
with $Q_0 \le c_5 e^{k/2} \psi_d(t)^{-d}$, 
must exceed $O(e^{-3k/2})$, an event whose 
probability is $O(e^{2k} \psi_d(t)^{-d})$.   

\smallskip
\noindent
{\bf Step I.}
The $\bP$-martingale $(M_u,\cF_u)$ is non-negative, continuous, hence converges
$\bP$-almost surely to a finite limit $M_\infty$. Further, 
$M_u = M_0 + W_{\langle M\rangle_u}$ for 
a standard Brownian motion $(W_s, s \ge 0)$, time changed by the 
quadratic variation $\langle M \rangle_u$ (e.g. \cite[Theorem 3.4.6, Problem 3.4.7]{KS}).
In particular, having a.s. finite $M_\infty$ implies the same for 
$\langle M \rangle_\infty$. In view of Lemma \ref{embedding}, 
for any $i \in \N$, 
\begin{align*}
\langle M\rangle_i &\ge \sum_{j=1}^i\bE[M_j^2-M_{j-1}^2|\cF_{j-1}] 
\ge 2 \wt{c} \, \sum_{j=0}^{i-1} \kappa_j^2 M_j^{2-2/d} \,,
\end{align*}
with the right inequality due to Lemma \ref{recursion} (for $\alpha=2 > 2/d$).
Since $\psi_d(\infty)=\infty$, it then follows that 
\begin{align*}
\liminf_{i\to\infty} \, \frac{\langle M \rangle_i}{\psi_d(i)} \ge \, 2 \wt{c} 
\, \liminf_{i\to\infty} \, \frac{1}{\psi_d(i)} \sum_{j=0}^{i-1} \kappa_j^2 M_{j}^{2-2/d}
=2 \wt{c} \, M_\infty^{2-2/d} \,. 
\end{align*}
We thus see that with probability one, if $M_\infty > 0$ then 
$\langle M \rangle_\infty = \infty$, out of which we deduce 
that necessarily $M_\infty=0$. The a.s. convergence to zero of $M_t$ 
allows us in turn to deduce that for any $u \ge 0$ and $z>0$, 
\begin{equation}\label{eq:bd-max-mg}
\bP(\sup_{t \ge u} \{M_t\} \ge z | \cF_u) = \frac{M_u}{z} \wedge 1 \,.
\end{equation}
Indeed, in case $M_u=0$ the martingale condition implies that a.s. 
$M_t \equiv 0$ for all $t \ge u$, whereas for $M_u \in (0,z)$ we get 
\eqref{eq:bd-max-mg} by applying for example \cite[Problem 1.3.28(i)]{KS}.

Turning to bound the left-sum in \eqref{reduce-to-tail}, note that subject to 
$\bI\{\tau_k>r\}$, the probability of $\{\wt{M} \ge e^k \}$ given $\cF_r$ is 
precisely the \abbr{lhs} of \eqref{eq:bd-max-mg} for $z=e^k$ and $u=r$.
With the unconditional probability given by \eqref{eq:bd-max-mg} with $u=0$, it thus follows 
that 
\begin{equation}\label{eq:ch-meas-bd}
\bP_{\{x\}}(\tau_k>r|{{\wt{M}\ge e^k}})=\bE_{\{x\}} \Big[\frac{M_r}{M_0} \, \bI(\tau_k>r)\Big]
\le \bE_{\{x\}} \Big[\frac{M_r}{M_0} \, \bI(T_k>r) \Big] \,,
\end{equation}
where $T_k = \inf \{ i \in \N : \pi^{(i)}(S_i) \ge e^k\}$ is  
the discrete-time analog of $\tau_k$ of \eqref{def:st-tm} (hence
necessarily $T_{k} \ge \tau_k$). Next note that the \abbr{rhs} 
of \eqref{eq:ch-meas-bd} equals $\wh{\bP}(T_k>r)$ for the 
martingale change of measure 
$$
\frac{{\sf d}\wh{\bP}}{{\sf d}\bP} (S_0,\ldots,S_r) = \frac{\pi^{(r)}(S_r)}{\pi^{(0)}(x)} \,.
$$ 
The measure $\wh{\bP}$ is thus given by the time-in-homogeneous 
Doob $h$-transform of the evolving sets process, for 
$h(t,A)=\pi^{(t)}(A)$, namely the measure
corresponding to the transition kernel $\wh{K}(\cdot,\cdot)$
of \eqref{eq:whK-def}. That is, $\wh{\bP}$ is the law of the 
conditioned (size-biased) evolving set of Definition \ref{defn-DF-evol}. 
 
\noindent
{\bf Step II.}
Under $\wh{\bP}$ with probability one 
$S_i$ are non-empty and  
$Y_i:=\pi^{(i)}(S_i)^{-1/2}\mathbb{I}_{\{T_k>i\}}$ finite, whereby from
Markov's inequality and \eqref{eq:ch-meas-bd} we deduce that for any $k$,
\begin{align}
\bP_{\{x\}}(\tau_k>r|{{\wt{M}\ge e^k}}) &\le \wh{\bP}_{\{x\}}(T_k>r)
\nonumber \\
&= \wh{\bP}_{\{x\}}(Y_r>e^{-k/2}) \le e^{k/2} \wh{\bE}_{\{x\}}(Y_r) \,.
\label{eq:ch-meas-bd2}
\end{align}
Further, by Lemma \ref{recursion} with $\alpha=1/2$ and $c=\wt{c}/8>0$, we have that  
$\wh{\bP}$-a.e. if $Y_i>0$, namely $T_k>i$, then  
\begin{align*}
\wh{\bE}_{\{x\}} \big(Y_{i+1}|Y_i\big)&=\bE_{\{x\}} \Big(\frac{\pi^{(i+1)}(S_{i+1})^{1/2}\mathbb{I}_{\{T_k>i+1\}}}{\pi^{(i)}(S_i)}|Y_i\Big)\\
&\le Y_i^2 \bE_{\{x\}} \big(\pi^{(i+1)}(S_{i+1})^{1/2}|Y_i\big)
\le Y_i(1- 2 c \kappa_i^2 Y_i^{4/d})\,. 
\end{align*}
Note that either $Y_i=0$, that is $\{T_k \le i\}$, in which case 
necessarily $Y_{i+1}=0$ and the preceding inequality holds, or else
by definition $Y_i>e^{-k/2}$. Thus, $\wh{\bP}$-a.e. for all $i$ and $Y_i$,
\begin{align}
\wh{\bE}_{\{x\}}(Y_{i+1}|Y_i)\le  Y_i\big[1- 2 c \kappa_i^2(Y_i^{4/d}\vee e^{-2k/d})\big]. 
\label{sup-MG}
\end{align}
Recall \cite[Lemma 12]{MP2} that $E[2 Z f(2Z)] \ge (E Z) f(E Z)$ for any
$Z \ge 0$ and non-decreasing $f: \R_+ \mapsto \R_+$. In particular, with 
$l_i:=\wh{\bE}_{\{x\}}(Y_i)$ and 
$f(y)=(y/2)^{4/d} \vee e^{-2k/d}$, we deduce upon taking the expectation 
of \eqref{sup-MG} that
\begin{align}\label{eq:li-mon}
l_{i+1}\le l_i-c \kappa_i^2 l_i f(l_i)\le l_i e^{-c \kappa_i^2 f(l_i)}\,. 
\end{align}
With $f(l_i)$ strictly positive it thus follows 
that either $l_i=0$, or else
\begin{equation}\label{eq:lbd-li}
\int_{l_{i+1}}^{l_i} \frac{dz}{zf(z)} \ge \frac{1}{f(l_i)} \int_{l_{i+1}}^{l_i} 
\frac{dz}{z} = \frac{1}{f(l_i)} \log \frac{l_i}{l_{i+1}} \ge c \kappa_i^2 \,.
\end{equation}
Hence, if $l_r>0$ then by \eqref{eq:li-mon}, $l_i>0$ for $i<r$ and
summing \eqref{eq:lbd-li} over $0 \le i < r$, yields 
\begin{align}
c \psi_d(r) \le \int_{l_r}^{\infty} (2^{4/d}z^{-1-4/d})\wedge (e^{2k/d}z^{-1})dz  \label{bound-lr}
\end{align}
(which trivially 
holds also when $l_r=0$).
We proceed to rule out having $l_r > 2e^{-k/2}$. Indeed, in that case
we get from \eqref{bound-lr} that 
\begin{align*}
c \psi_d(r)\le\int_{l_r}^\infty 2^{4/d}z^{-1-4/d}dz=2^{4/d}(d/4)l_r^{-4/d},
\end{align*}
whereby $l_r\le c'\psi_d(r)^{-d/4}$ for $c'=2(4d/c)^{d/4}$. As 
$k<L$, this yields in view of \eqref{eq:r-choice} and our choice of $L$ that
\begin{align*}
\vep^{-1} \psi_d(t)^{-d/4} \le e^{-(L-1)/2} < 2 e^{-k/2}
<l_r\le c'\psi_d(r)^{-d/4} \le c' 3^{d/4} \psi_d(t)^{-d/4} \,, 
\end{align*}
yielding a contradiction when $\vep= (1/c') 3^{-d/4}$.  
Taking hereafter such $\vep$ we thus have that $l_r \le 2e^{-k/2}$ in which case
\eqref{bound-lr} yields 
\begin{align*}
c \psi_d(r) \le e^{2k/d} \int_{l_r}^{2e^{-k/2}} \frac{dz}{z} +
2^{4/d} \int_{2e^{-k/2}}^{\infty} \frac{dz}{z^{1+4/d}} 
\le e^{2k/d}\big(\log (2 e^{-k/2}/l_r) + c_0\big),
\end{align*}
for some finite $c_0=c_0(d)$. That is, for $c_1 = 2 e^{c_0}$ finite,
$$
l_r \le c_1 e^{-k/2} \exp\big\{- c \psi_d(r)e^{-2k/d}\big\} \,.
$$
Plugging this bound in the \abbr{rhs} of \eqref{eq:ch-meas-bd2}, we bound
the left sum in \eqref{reduce-to-tail} after change of variable 
$s=e^{-2k/d}\psi_d(r)$, by
\begin{align}
\sum_{k=k_0}^{L-1}e^{-k}\bP_{\{x\}}(\tau_k>r&|{{\widetilde{M}\ge e^k}}) \le c_1 
\sum_{k=k_0}^{L-1} e^{-k} \exp\{-c \psi_d(r)e^{-2k/d}\}\nonumber\\
&\le c_2\int_0^\infty e^{-c s}(s/\psi_d(r))^{d/2}s^{-1}ds\le  c_3\psi_d(r)^{-d/2}, \label{bound-1}
\end{align}
for some finite constants $c_j=c_j(d,\gamma)$, $j=2,3$.

\smallskip
\noindent
{\bf Step III.}
Moving next to bound the right sum in \eqref{reduce-to-tail}, conditioning on 
$\{\cF_{\lceil \tau_k \rceil}, \tau_k \le r\}$ 
we have by the strong Markov property at
$\lceil \tau_k\rceil$ that,
$$
\wt{S}_{i}:=S_{\lceil \tau_k\rceil+i}, \quad i \ge 0, 
$$
is an evolving set process for conductances 
$\wt{\pi}^{(i)}(\cdot):=\pi^{(\lceil \tau_k\rceil +i)}(\cdot)$, with which we also associate
$$
\wt{\kappa}_{i} := \kappa_{\lceil \tau_k\rceil +i},\qquad
\wt{\psi}_d(i):=\psi_d(\lceil \tau_k\rceil +i)-\psi_d(\lceil \tau_k\rceil)\,.
$$
Note that if $k \ge k_0$ then $\tau_k>0$ and hence
$M_{\tau_k}=e^k$ whenever $\tau_k < \infty$.  
Thus, from \eqref{eq:bd-max-mg} at 
the stopping time $u=\tau_k \le r$, we deduce that 
$$
\bP_{\{x\}}(E_k | \tau_k \le r) = \bP_{\{x\}}
(\tau_{k+1}=\infty| \tau_k \le r) = 1 - e^{-1} \,.
$$
Consequently, for $c_4=1/(1-e^{-1})$, $k \ge k_0$
and any 
$\cF_{\lceil \tau_k \rceil + i}$-stopping time
$$
\sigma := \inf\{i \ge 0: \wt{\pi}^{(i)}(\wt{S}_{i}) > 
e^{a_{i+1}} \} \,,
$$
with $a_{t-\lceil \tau_k \rceil}=-\infty$, one has that
\begin{align}
\bP_{\{x\}}(T_k' > t \,|\, E_k, \tau_k\le r)
&= c_4
\bP_{\{x\}}(T_k' > t, 
\tau_{k+1}=\infty \,|\,  \tau_k\le r) \nonumber \\
& \le c_4
\bP_{\{x\}}(\sigma <
 t - \lceil \tau_k \rceil,
\tau_{k+1}=\infty \,|\, 
 \tau_k\le r) \,.
 \label{eq:cond-bd}
\end{align}
In particular, we shall employ \eqref{eq:cond-bd} for 
the non-increasing $a_{i}$ such that 
\begin{align}
\label{a_i-defn}
\frac{\tilde{c}}{4} \wt{\kappa}_{i}^{2}
=2 \int_{a_{i+1}}^{a_i} e^{2z/d} dz
\,, \qquad 0 \le i < t-\lceil \tau_k\rceil\,.
\end{align} 
To this end, we first show that 
$(Q_{i \wedge \sigma},\cF_{\lceil \tau_k \rceil + i})$,
$i < t - \lceil \tau_k \rceil$ 
is a super-martingale, for  
\begin{align*}
Q_{i}:=e^{-2a_{i}} \wt{Y}_i \,,\qquad
\wt{Y}_{i}:=\wt{\pi}^{(i)}(\wt{S}_{i})^{1/2}\bI(\tau_{k+1}>\lceil \tau_k \rceil +i) \bI(\tau_k \le r) \,.
\end{align*}
Indeed, applying Lemma \ref{recursion} (for $\alpha=1/2$), 
to the evolving process $\{\wt{S}_i\}$, 
if $\tau_k \le r$ and 
$\tau_{k+1} >\lceil \tau_k \rceil +i$ then
\begin{align*}
 \bE_{\{x\}} [\wt{Y}_{i+1} |
\cF_{\lceil\tau_k\rceil +i}]
\le \wt{Y}_i \big(1-\frac{\tilde{c}}{4} \wt{\kappa}_{i}^{2} {\wt{Y}_i}^{-4/d}\bI_{\{\wt{Y}_{i}>0\}}\big) \,.
\end{align*}
This inequality trivially holds if either
$\{\tau_{k+1} \le \lceil \tau_k \rceil + i\}$
or $\{\tau_k>r\}$
(whereby both sides are zero), yielding that 
for $i<t-1-\lceil \tau_k \rceil$
\begin{align}
\bE_{\{x\}} [Q_{i+1}|\cF_{\lceil\tau_k\rceil +i}]
\le Q_i \exp\big\{ 2(a_{i}-a_{i+1}) -
\frac{\tilde{c}}{4} 
\wt{\kappa}_i^{2} (\wt{Y}_i)^{-4/d} \bI_{\{\wt{Y}_{i}>0\}}\big\}. 
\label{sup-MG-again}
\end{align}
Recall that our choice of $a_i$ in (\ref{a_i-defn}), 
implies that 
\begin{align}
\frac{\tilde{c}}{4} \wt{\kappa}_{i}^{2}
\ge 2 (a_i - a_{i+1}) e^{2 a_{i+1}/d}
\ge 2 (a_i-a_{i+1})(\wt{Y}_{i})^{4/d} \,,
\label{mon-a_i}
\end{align}
when $\wt{Y}_{i} \le e^{a_{i+1}/2}$. Thus,
the exponent on the \abbr{rhs} of
(\ref{sup-MG-again}) is non-positive 
when both $i<\sigma$ and $\wt{Y}_{i}>0$, 
in which it follows from (\ref{sup-MG-again}) that 
$$
\E_{\{x\}} (Q_{(i+1) \wedge \sigma}
|\cF_{\lceil \tau_k\rceil +i}) 
\le Q_{i \wedge \sigma}\,.
$$
As this inequality trivially holds with equality 
when $i \ge \sigma$, as well as when $\wt{Y}_i=0$
(for then also $\wt{Y}_{i+1}=0$), we have 
the claimed super-martingale property. 
\newline
Now, since $\wt{Y}_i < e^{(k+1)/2}$, if
$\tau_k \le r$ then by \eqref{a_i-defn},
\begin{align*}
Q_0 \le e^{-2a_{0}} e^{(k+1)/2}
\le c_5 e^{k/2}(\psi_d(t)-\psi_d(r))^{-d}\,,
\end{align*}
for some $c_5=c_5(d,\gamma)$
finite. Further, by the definition of $\sigma$, if
$i=\sigma < t - \lceil \tau_k \rceil$
then $\wt{S}_i$ is non-empty, hence 
$\wt{\pi}^{(i)}(\wt{S}_i)^{1/d} \ge \wt{\kappa}_i$ 
(see \eqref{eq:kappa-bd}). It then follows 
from \eqref{a_i-defn} that 
$$
e^{2 a_i/d} = e^{2 a_{i+1}/d} + 
\frac{\tilde{c}}{4d} \wt{\kappa}_{i}^{2}
\le e^{2 a_{i+1}/d} + 
\frac{\tilde{c}}{4d} \wt{\pi}^{(i)}(\wt{S}_i)^{2/d}
$$
which by definition of $\sigma$ implies that also 
\begin{equation}\label{eq:bd-sigma}
\wt{\pi}^{(i)}(\wt{S}_i) \ge c_6 e^{a_i} 
\end{equation}
for $c_6 := (1+\tilde{c}/(4d))^{-d/2}$ positive.  
In case $\tau_{k+1}=\infty$ it further 
suffices to consider only those $i \ge 0$ for which 
the \abbr{rhs} of \eqref{eq:bd-sigma} is at most
$e^{(k+1)}$, implying in turn that 
(when also $\tau_k \le r$),
$$
Q_i = e^{-2 a_i} \wt{\pi}^{(i)}(\wt{S}_i)^{1/2}
\ge c_6^2 (c_6 e^{a_i})^{-3/2} 
\ge c_6^2 e^{-3(k+1)/2} \,.  
$$
In conclusion, when $\tau_k \le r$, 
\begin{align*} 
\{ \sigma < t - \lceil \tau_k \rceil, \tau_{k+1}=\infty \}
\subseteq
\{ \sigma < t - \lceil \tau_k \rceil, 
Q_{\sigma} \ge c_6^2 e^{-3(k+1)/2} \} \,.
\end{align*}
Applying Doob's optional stopping 
to the non-negative super-martingale\\ 
$\{Q_{i \wedge \sigma}\}$ we further bound 
the \abbr{rhs} of \eqref{eq:cond-bd} by 
\begin{align*}
c_4 \bP_{\{x\}} (
\sigma < t - \lceil \tau_k \rceil, Q_{\sigma} 
\ge c_6^2 e^{-3(k+1)/2} | \tau_k \le r) 
& \le c_7 e^{3k/2} \bE_{\{x\}} (Q_0 | \tau_k \le r) 
\\
& \le c_8 e^{2 k} (\psi_d(t)-\psi_d(r))^{-d}  \,,
\end{align*}
for some finite $c_j(d,\gamma)$, $j=7,8$. In view of 
\eqref{eq:cond-bd}
the right sum of (\ref{reduce-to-tail}) is thus bounded by
\begin{align}
\sum_{k=k_0}^{L-1}e^{-k}\bP_{\{x\}}
(T'_k > t |E_k, \tau_k\le r)\le c_9 e^{L}(\psi_d(t)-\psi_d(r))^{-d}.
\label{bound-2}
\end{align}
For our choice of $L$, the bound \eqref{unbdd-beta} follows from 
(\ref{reduce-to-tail}), (\ref{bound-1}) and (\ref{bound-2}).
\qed

\medskip
\section{Proofs of Propositions \ref{csrw}, \ref{percol} and \ref{prop-df}}\label{sec-percol}

\medskip
\noindent
{\emph{Proof of Proposition \ref{csrw}.}} 
Let $\{\tau_j\}$ be a collection of i.i.d $\exp(2)$ random  variables. We
simulate the \abbr{CSRW} using $T_k:=\sum_{j=1}^k \tau_j$ as our successive Poisson clocks and independently designate that each time $T_k$ the clock rings,  with probability $(1/2)$ the walk $Y_t$ stays put, and 
with probability $(1/2)$ it makes a jump according to 
the given edge conductances at time $T_k$. 
By the thinning property of the Poisson process, the simulated 
process $t \mapsto Y_t$ is the \abbr{CSRW} of Definition \ref{defn:csrw}. 
On the other hand, the sampled process $X_k=Y_{T_k}$ has the law of
$(1/2)$-lazy discrete time random walk on $\bG$, with time-varying edge conductances $\pi^{(T_k)}(x,y)$ that are in $[C_1^{-1},C_1]$ for 
every realization $\omega$ of $\{T_k\}$. Consequently, 
denoting $N_t=\max\{k\in\mathbb{N}: T_k\le t\}$, a Poisson process
of rate $2$, we have as in Corollary \ref{cor:u-ellip} 
that for some $C_2=C_2(d,C_1)>0$, any $t \ge s \ge 0$ and all $\omega$, 
\begin{align*}
\psi^\omega_d(t)-\psi^\omega_d(s)\ge C_2(N_t-N_s).
\end{align*}
From Definition \ref{rcll-eff-non-dec} of the effectively non-decreasing 
\abbr{RCLL} conductances  $t\mapsto\pi^{(t)}(x)$ and \eqref{eq:HK-ubd-imp}, 
for a.e. $\omega=\{T_k\}$ there exists 
$c^\omega_\ast=c^\omega_\ast(d,C_1)$ finite,
such that the quenched heat-kernel bound  
\begin{align}\label{eq:quenched-hk-ubd}
P^\omega(s,x;t,y)\le {c^\omega_\ast}\big[ e+C_2(N_t-N_s)
\big]^{-d/2}=:{c^\omega_\ast\phi(N_t-N_s),}
\end{align}
applies for the transition probabilities $P^\omega(s,x;t,y)$ 
of the \abbr{CSRW} $\{Y_t\}$.
With $\phi(\cdot)$ positive and decreasing on $\R_+$, we have that 
$$
\phi(N_t-N_s) \le 
\phi(0) {\mathbb I}_{\{N_t - N_s \le t-s\}} + \phi(t-s) \,,
$$
and consequently 
\begin{align}
\int_0^\infty \bE[{(c_\ast^\omega)^{-1}P^\omega(0,x;t,y)}] \mathrm{d} t 
 \le \phi(0) \int_0^\infty \P(N_t \le t) \mathrm{d} t 
+ \int_0^\infty \phi(t) \mathrm{d}t 
\label{eq:loc-tim-csrw}
\end{align}
is finite. Thus, by Fubini's theorem  
$\int_0^\infty (c_\ast^\omega)^{-1}P^\omega(0,x;t,y) dt$ is 
finite for a.e. $\omega$,
which together with the finiteness of $c_\ast^\omega$ implies that 
$\int_0^\infty P^\omega(0,x;t,y) dt$ is finite. That is, 
starting at any non-random $x \in V$
we have 
a finite 
total local time for the \abbr{csrw} 
at any $y \in V$. Hence, for a.e. $\omega$ the 
sampled process at jump times $\{Y_{T_k}\}$, 
visits every $y \in V$ only finitely often.
\qed

\medskip
\noindent
{\emph{Proof of Proposition \ref{percol}.} } 
In view of Theorem \ref{thm-main} 
with $V=\mathcal{V}(\bD_0)$
 and Remark \ref{rmk:trans}, with $d/2>1$, {
we have the stated claim upon showing that 
for any $\theta=\theta_{\text{iso}}>0$, there exists some $T=T(\omega, \theta)<\infty$ and constant $c'(\theta,d)>0$ such that the 
isoperimetric growth function satisfies 
\begin{equation}\label{eq:grow-trans}
\psi_d(t) \ge c't^{1-\theta(1-1/d)} \,, \qquad \forall t \ge T, {{\quad P^\omega\text{-a.s.}}} 
\end{equation}
(as then $\psi_d(t)^{-d/2}$ would be summable upon taking $\theta$ sufficiently small.)}

To this end, let $\bD_u^{\ell}$ denote the vertices of  
$\bD_u \cap [-\ell,\ell]^d$ and recall that 
starting at $x=0$ we have that the evolving set
$S_u \subseteq \bD_u^u$ (because the \abbr{srw}
has at most linear growth in each direction).
Here $\pi^{(u)}(x,y) \in \{0,1\}$ so we are just counting edges. Further,
with all degrees of vertices of $\bD_u$ within $[1,2d]$, we replace
$\pi^{(u)}(A)$ by the size $|A|^{(u)}$ of $A \cap \bD_u$ with 
$|\partial A|^{(u)}=\pi^{(u)}(A,A^c)$. 
By \cite[Theorem 1.2]{PRS} (together with Borel-Cantelli lemma), the 
unique infinite percolation cluster $\bD_0$ of \cite{PRS}
satisfies the following isoperimetric inequality for some $c=c(\theta)>0$ and all
$l \ge l_0(\omega,\theta)$ large enough
\begin{align}
\inf_{A\subseteq \bD_0^l, |A|\le|\bD_0^l|/2}
\; \Big\{ \frac{|\partial_{\bD_0^l}A|}{|A|^{(d-1)/d}} \; \Big\} 
\ge cl^{-\theta(1-1/d)} {{\quad P^\omega\text{-a.s.}}}
\label{perc-isop}
\end{align}
Moreover, since to $\bD_u$ we only add edges and no new vertices, clearly 
$|\partial A|^{(u)}\ge|\partial A|^{(0)}$, and $|A|^{(0)} = |A|^{(u)}$, with  
the inequality (\ref{perc-isop}) holding uniformly for all $\{\bD_u\}$. 
{Applying (\ref{perc-isop}) to sets $S_u\subseteq \bD_u^{2u}$, we have that
\begin{align*}
\kappa_u \ge c'u^{-\theta(1-1/d)},
\end{align*}
yielding all $t \ge 2 T$, the claimed growth of \eqref{eq:grow-trans},
\begin{equation*}
\qquad\qquad\qquad
\psi_d(t) \ge \sum_{u=T}^{t-1} \kappa_u^2 \ge c'^2 (t-T) t^{-\theta(1-1/d)}\ge \frac{c'^2}{2}t^{1-\theta(1-1/d)}.
\qquad\qquad\qquad\qquad\qquad{\qed}
\end{equation*}
}

\bigskip
\noindent
\emph{Proof of Proposition \ref{prop-df}.}
With (a) and (b) trivially holding at $t=0$, we proceed by induction on $t$. 
Specifically, we assume that both (a) and (b) hold for some $t \ge 0$.
Then, with $\mathbf{S}_t=(S_0,\ldots,S_t)$, by the definition of 
$P(\cdot;\cdot)$ and $P^*(\cdot;\cdot)$, our 
hypothesis of (b) holding for $t$ implies that
for any $v \in B$ such that $\pi^{(t)}(S_t,v)>0$, 
\begin{align}
\mathbb{P}^*_{x,\{x\}} & (X_{t+1}=v, S_{t+1}=B|\mathbf{S}_t) \notag \\
&=\sum_{w\in S_t}\mathbb{P}^*(X_{t+1}=v, S_{t+1}=B|X_t=w,\mathbf{S}_t)\mathbb{P}^*_{x,\{x\}}(X_t=w|\mathbf{S}_t) \notag \\
&=\sum_{w\in S_t}\frac{P(t,w;t+1,v)K(t,S_t;t+1,B)\pi^{(t+1)}(v)}{\pi^{(t)}(S_t,v)}\frac{\pi^{(t)}(w)}{\pi^{(t)}(S_t)}  \notag \\
&=\frac{\pi^{(t+1)}(v)}{\pi^{(t)}(S_t)}\frac{\sum_{w\in S_t}\pi^{(t)}(w)P(t,w;t+1,v)}{\pi^{(t)}(S_t,v)}K(t,S_t;t+1,B) \notag \\
&=\frac{\pi^{(t+1)}(v)}{\pi^{(t)}(S_t)}K(t,S_t;t+1,B)\,.
\label{eq:joint-law-df}
\end{align}

By Definition \ref{defn-DF-evol}, the conditioned evolving set is such 
that $X_{t+1} \in S_{t+1}$ so
the \abbr{lhs} of \eqref{eq:joint-law-df} is zero when $\pi^{(t)}(S_t,v)=0$.
Consequently, summing in \eqref{eq:joint-law-df} over $v\in B$ we find that
\begin{align}\label{eq:law-evol-df}
\mathbb{P}^*_{x,\{x\}}(S_{t+1}=B|\mathbf{S}_t)
=\frac{\pi^{(t+1)}(B)}{\pi^{(t)}(S_t)}K(t,S_t;t+1,B) = \widehat{K}(S_t,B)\,,
\end{align}
and thereby verify that our claim (a) extends up to $t+1$. Further, 
the ratio of \eqref{eq:joint-law-df} and \eqref{eq:law-evol-df} 
results with  
\begin{align*}
\mathbb{P}^*_{x,\{x\}}&(X_{t+1}=v | S_{t+1}=B, \mathbf{S}_t)=\frac{\pi^{(t+1)}(v)}{\pi^{t+1}(B)} \,,
\end{align*}
which amounts to the claimed property (b) at $t+1$.
\qed

\end{section}

\end{document}